\documentclass[a4paper]{article}
\usepackage{mathptmx}
\usepackage{amscd,amssymb,amsmath,amsthm}

\makeatletter

\renewcommand{\@seccntformat}[1]
{{\csname the#1\endcsname}.\hspace{0.3em}}

\renewcommand{\section}{\@startsection
{section}
{1}
{0mm}
{-1.5\baselineskip}
{\baselineskip}
{\bfseries\normalsize}}

\renewcommand{\subsection}{\@startsection
{subsection}
{2}
{0mm}
{-\baselineskip}
{0.5\baselineskip}
{\normalsize\itshape}}

\renewcommand{\subsubsection}{\@startsection
{subsubsection}
{3}
{0mm}
{-.5\baselineskip}
{-2mm}
{\normalsize\itshape}}

\makeatother

\theoremstyle{plain}
\newtheorem*{theorem*}{Theorem}
\newtheorem{theorem}{Theorem}[section]
\newtheorem{lemma}{Lemma}[section]

\newtheorem{prop}[lemma]{Proposition}

\newtheorem*{corollary*}{Corollary}
\newtheorem*{openq}{Open Question}

\theoremstyle{definition}
\newtheorem*{defin*}{Definition}

\theoremstyle{remark}
\newtheorem{remark}{Remark}[section]

\newtheorem*{quest*}{Question}

\DeclareMathAlphabet{\matheur}{U}{eur}{m}{n}
\DeclareMathAlphabet{\matheus}{U}{eus}{m}{n}
\DeclareMathAlphabet{\matheuf}{U}{euf}{m}{n}

\numberwithin{equation}{section}

\newcommand{\abs}[1]{\left\lvert#1\right\rvert}
\newcommand{\norm}[1]{\left\lVert#1\right\rVert}

\DeclareMathOperator{\dist}{dist}
\DeclareMathOperator{\inj}{inj}
\DeclareMathOperator{\spt}{spt}
\DeclareMathOperator{\reg}{reg}

\DeclareMathOperator{\grad}{grad}

\DeclareMathOperator{\hess}{Hess}

\begin{document}

\author{Gerasim  Kokarev
\\ {\small\it School of Mathematics, University of Leeds}
\\ {\small\it Leeds, LS2 9JT, United Kingdom}
\\ {\small\it Email: {\tt G.Kokarev@leeds.ac.uk}}
}

\title{Remarks on the spectra of minimal hypersurfaces in the hyperbolic space}
\date{}
\maketitle

\begin{abstract}
\noindent
We compute the Laplacian spectra of singular area-minimising hypersurfaces in the hyperbolic space with prescribed asymptotic data. We also obtain similar results in higher codimension, and  explore related extremal properties of the bottom of the spectrum.
\end{abstract}

\medskip
\noindent
{\small

}

%\tableofcontents

\section{Introduction}
\label{intro}

\subsection{Main result}
In~\cite{LMMV} the authors study the spectrum of the Laplace operator on minimal submanifolds in space forms. In particular, among a number of results, they prove the following statement.
\begin{theorem}
\label{t0}
Let $\Sigma^m\subset\mathbf H^{n+1}$ be a proper $C^2$-smooth minimal submanifold  in the hyperbolic space $\mathbf H^{n+1}$ that extends to a $C^2$-smooth manifold with boundary in the closure ${\mathbf {\bar H}^{n+1}}$, obtained by adding the sphere at infinity $\mathbf H^{n+1}\cup S^n_\infty(\mathbf H^{n+1})$. Then the spectrum of the Laplace operator on $\Sigma^m$ coincides with the spectrum of the Laplace operator on the $m$-dimensional hyperbolic space $\mathbf H^m$, that is the interval $[(m-1)^2/4,+\infty)$.
\end{theorem}

To our knowledge, when $m=n\geqslant 7$ non-trivial examples of minimal submanifolds that satisfy the hypotheses of Theorem~\ref{t0} occur mostly as minimal graphs, whose existence is related to very specific asymptotic data. The first purpose of this paper is to prove the version of this statement for singular area-minimising hypersurfaces, whose existence is always guaranteed for arbitrary asymptotic data by the results in ~\cite{An82}.

\begin{theorem}
\label{t1}
Let $\Gamma^{n-1}\subset S^n_\infty(\mathbf H^{n+1})$ be a closed oriented submanifold in the sphere at infinity, and let $\Sigma^n\subset \mathbf H^{n+1}$ be an area-minimising locally rectifiable $n$-current asymptotic to $\Gamma^{n-1}$, that is such that $\overline{\spt \Sigma^n}\cap S^n_\infty(\mathbf H^{n+1})=\Gamma^{n-1}$. Then the spectrum of the Laplace operator on $\Sigma^n$ is the interval $[(n-1)^2/4,+\infty)$.
\end{theorem}

Above by the spectrum of the Laplace operator on $\Sigma^n$ we mean the spectrum of the Friedrichs extension of the Laplacian defined on smooth functions with compact support inside the regular set $\reg\Sigma^n$. We refer to Section~\ref{prems} for the related background material and notation. The proof of Theorem~\ref{t1} builds on two ingredients. First, a comparison argument, reminiscent to the one in~\cite{Ko21,LMMV}, can be carried over to the singular setting to prove a version of McKean's bound for the bottom of the spectrum of singular minimal surfaces. This is an extension of closely related results in the smooth setting that go back to~\cite{CLY,Ma86}, see also~\cite{BM07}. Our main statement here is the comparison theorem for isoperimetric constants, Theorem~\ref{tc}, which might be of independent interest.

Second, to show that any value  above $(n-1)^2/4$ lies in the spectrum, we essentially rely on the boundary regularity results for area-minimising hypersurfaces due to Hardt and Lin~\cite{HL87}. This part of the argument uses specific behaviour at the boundary at infinity, and unlike Theorem~\ref{t0}, requires only $C^1$-smoothness up to boundary, see the discussion in Remark~\ref{r:c1}. Besides, related spectral properties hold for a more general class of locally rectifiable currents. More precisely, for a closed submanifold $\Gamma^{m-1}\subset S_\infty^n(\mathbf{H}^{n+1})$ denote by $\mathcal M(\Gamma^{m-1})$ the set of $m$-dimensional locally rectifiable currents $T^m$ in $\mathbf{H}^{n+1}$ whose regular set $\reg_1 T^m$, formed by points whose vicinity in $\spt T^m$ is a $C^1$-smooth submanifold, satisfies the following conditions:
\begin{itemize}
\item[(i)] the complement $\spt{T^m}\backslash \reg_1 T^m$ is bounded in $\mathbf{H}^{n+1}$, that is contained in some hyperbolic ball $\mathbf{B}_r\subset\mathbf{H}^{n+1}$;
\item[(ii)] the union $\reg_1 T^m\cup\Gamma^{m-1}$ is a $C^1$-smooth submanifold with boundary that meets the boundary sphere $S_\infty^n(\mathbf{H}^{n+1})$ orthogonally.
\end{itemize}
Below by $\lambda_*(T^m)$ we denote the fundamental tone of $T^m$, see Section~\ref{prems} for a precise definition. Our main observation is that the quantity $\lambda_*(T^m)$ satisfies a certain sharp bound when $T^m$ ranges in $\mathcal M(\Gamma^{m-1})$ and area-minimising currents that lie in this set are natural maximisers for $\lambda_*$.
\begin{theorem}
\label{t2}
Let $\Gamma^{m-1}\subset S_\infty^n(\mathbf{H}^{n+1})$ be a closed submanifold in the sphere at infinity. Then for any $T^m\in\mathcal M(\Gamma^{m-1})$ the fundamental tone satisfies the inequality 
\begin{equation}
\label{bs:bound}
\lambda_*(T^m)\leqslant\frac{1}{4}(m-1)^2.
\end{equation}
Further, there exists a complete area-minimising locally rectifiable $m$-current $\Sigma^m\in\mathcal M(\Gamma^{m-1})$ such that $\lambda_*(\Sigma^m)$ saturates inequality~\eqref{bs:bound}. In addition, the Laplacian spectrum of $\Sigma^m$ is the interval $[(m-1)^2/4,+\infty)$.
\end{theorem}
Note that, since any complete Riemannian manifold can be properly isometrically embedded into the hyperbolic space $\mathbf{H}^{n+1}$ for a sufficiently large $n$, via an embedding into a horosphere, some hypotheses on the asymptotic behaviour of submanifolds in $\mathbf{H}^{n+1}$ are necessary for any bound for the fundamental tone to hold.  We believe that the choice of $\mathcal M(\Gamma^{m-1})$ is very natural in view of the regularity theory for the asymptotic Plateau problem, and our arguments in Section~\ref{proofs} underline an explicit relation to the spectral properties of the asymptotic cone. Although our approach to the proof of inequality~\eqref{bs:bound} is rather elementary, unlike other papers in the literature, it does not use any volume growth assumptions or curvature behaviour at infinity, see~\cite{LMMV} and references therein. 

The area-minimising $m$-current in Theorem~\ref{t2} is the so-called Anderson solution, whose regularity is studied by Lin in~\cite{L89b}. The fact that its bottom of the spectrum equals the maximal value is an immediate consequence of the version of McKean's bound mentioned above. In particular, by McKean's bound inequality~\eqref{bs:bound} is saturated by any stationary current $\Sigma^m\in\mathcal M(\Gamma^{m-1})$. However, there are also not necessarily stationary currents in $\mathcal M(\Gamma^{m-1})$, for example, hyperbolic cones, which maximise the fundamental tone, see Remark~\ref{rem:cones}. Understanding the phenomena responsible for the equality in~\eqref{bs:bound} seems to be a more subtle problem, which we plan to address in future work. Mention that the occurrence of  stationary currents that saturate inequality~\eqref{bs:bound} is reminiscent to the role of minimal surfaces in the classical extremal eigenvalue problems, see~\cite{FS13,Ko14} and references therein.

Our argument in the proof of Theorem~\ref{t2} shows that for any stationary current $\Sigma^m$ in the class $\mathcal M(\Gamma^{m-1})$ its Laplacian spectrum is precisely the interval $[(m-1)^2/4,+\infty)$. However, the answer to the following principal question in full generality is unknown:
\begin{openq}
Let $\Sigma^m$ be an area-minimising (or more generally, stationary) current in $\mathbf{H}^{n+1}$ that is asymptotic to a smooth submanifold $\Gamma^{m-1}$, in the sense that $\overline{\spt \Sigma^m}\cap S^n_\infty(\mathbf H^{n+1})=\Gamma^{m-1}$. Does the spectrum of the Friedrichs extension of the Laplace operator on $\Sigma^m$ coincide with the interval $[(m-1)^2/4,+\infty)$?
\end{openq}
Our Theorem~\ref{t1} settles this question for  area-minimising locally rectifiable currents in co-dimension one, that is  when $m=n$. In higher co-dimension, as the discussion above shows,  the answer would be also positive, if developing boundary regularity theory, one can show that a stationary current $\Sigma^m$ lies in $\mathcal M(\Gamma^{m-1})$. For example, by results in~\cite{L89b}, so are area-minimising flat chains modulo two. Note that stationary currents that are $C^1$-smooth near and up to boundary automatically lie in the class $\mathcal M(\Gamma^{m-1})$, see Remark~\ref{r:c1}. In particular, this observation sharpens the hypotheses of Theorem~\ref{t0}.

\subsection{Content and organisation of the paper}
Section~\ref{prems} contains most of the background material. First, we introduce necessary terminology from geometric measure theory and recall principal existence and regularity results for the asymptotic Plateau problem. Then we discuss self-adjoint extensions of the Laplace operator defined on rectifiable currents. We end with the notion of the fundamental tone and its relationship with the bottom of spectrum. In Section~\ref{comp} we discuss the comparison theorem for isoperimetric constants of stationary currents. Here we follow closely the notation and computations in~\cite{Ko21}. Section~\ref{proofs} contains the proofs of Theorems~\ref{t1} and~\ref{t2}. We essentially focus on the latter, the proof of the former follows the line of argument used to compute the Laplace spectrum of an area-minimising current in Theorem~\ref{t2}. The principal part of the argument is based on the construction of suitable test-functions and neat estimates for the corresponding values of the Laplacian near the ideal boundary. In the last section we collect a few remarks. They are mostly consequences of the proof of Theorem~\ref{t2}, and explain certain points that have been made in the introduction.

%\smallskip
%\noindent
%{\em Acknowledgements.} 

\section{Preliminary material}
\label{prems}
\subsection{Notation and first results}
We start with recalling basic notation from geometric measure theory; our main references are~\cite{FM,LS}.

Let $M$ be an oriented complete smooth Riemannian manifold. By $\mathcal D^n$ we denote the space of smooth compactly supported $n$-forms on $M$. The dual space $\mathcal D_n$, equipped with the weak topology, is called the space of $n$-{\em dimensional currents}. If $\Sigma^n$ is an oriented $n$-dimensional submanifold of $M$ with compact closure and finite $n$-dimensional volume, then it defines an $n$-current $[\![\Sigma^n]\!]$ by the formula
$$
[\![\Sigma^n]\!](\varphi)=\int_{\Sigma^n}\varphi,\qquad\text{ ~where }\varphi\in\mathcal D^n.
$$
More generally, a {\em rectifiable $n$-current} is a functional of the form $\sum\mu_i[\![\Sigma_i]\!]$, where $\{\Sigma_i\}$ is a countable collection of mutually disjoint rectifiable sets such that the closure of $\cup_i\Sigma_i$ is compact, $\mu_i$ is a positive integer multiplicity function such that $\sum\mu_i\mathcal H^n(\Sigma_i)<+\infty$, and $\mathcal H^n$ stands for the $n$-dimensional Hausdorff measure. The space of rectifiable $n$-currents is denoted by $\mathcal R_n$. For a current $S=\sum\mu_i[\![\Sigma_i]\!]$ its support is the closure of the union $\cup_i\Sigma_i$,
$$
\spt S=\overline{\cup_i\Sigma_i}.
$$
The boundary of an $n$-dimensional current $S\in\mathcal D_n$ is the $(n-1)$-dimensional current $\partial S\in\mathcal D_{n-1}$ defined by
$$
\partial S(\varphi)=S(d\varphi), \qquad\text{ ~where }\varphi\in\mathcal D^{n-1}.
$$
By Stokes's Theorem, this definition agrees with the standard notion of boundary, if $S$ is a smooth oriented manifold with boundary. By $\mathbf I_n$ we denote the space of {\em integral $n$-currents}, that is $n$-currents $S$ such that $S\in\mathcal R_n$ and $\partial S\in\mathcal R_{n-1}$.

In the sequel we are interested in currents with non-compact support. By $\mathcal R_n^{\mathit{loc}}$ we denote the space of {\em locally rectifiable $n$-currents}, formed by $S\in\mathcal D_n$ such that for any $x\in M$ there exists $T\in\mathcal R_n$ such that $x\notin\spt(S-T)$. Similarly, by $\mathbf I^{\mathit{loc}}_n$ we denote the space of {\em locally integral $n$-currents} defined by the condition above with $T\in\mathbf I_n$.

There is a natural notion of {\em mass} defined on $\mathcal D_n$:
$$
\mathbf{M}(S)=\sup\{S(\varphi): \varphi\in\mathcal D^n, \sup_x\abs{\varphi(x)}^*\leqslant 1\},
$$
where $x$ ranges in $M$, and by $\abs{\varphi(x)}^*$ we mean the supremum of the values of the form $\varphi(x)$ on simple unit $n$-vectors. Similarly, we have the following function, 
$$
\norm{S}(U)=\sup\{S(\varphi): \varphi\in\mathcal D^n, \sup_x\abs{\varphi(x)}^*\leqslant 1, \spt\varphi\subset U\},
$$
defined on open subsets $U$ in $M$. If $S$ has finite mass, then $\norm{S}$ extends to a Radon measure, denoted by the same symbol. It is straightforward to see that if $S=[\![\Sigma^n]\!]\in\mathcal R_n$, where $\Sigma^n$ is an $n$-dimensional submanifold, then $\norm{S}$ is the volume measure on $\Sigma^n$, and $\mathbf{M}(S)$ is the volume of $\Sigma^n$. Finally, a current $S\in\mathcal{R}^{\mathit{loc}}_n$ is called {\em absolutely area-minimising}, if for any compact subset $K\subset M$ the inequality
$$
\mathbf{M}(S\llcorner K)\leqslant\mathbf{M}(T)
$$
holds for all $T\in\mathcal{R}_n$ such that $\partial(S\llcorner K)=\partial T$. Above by $S\llcorner K$ we mean the restriction to a compact set $K$, defined as $S\llcorner K(\varphi)=S(\chi_K\varphi)$, where $\chi_K$ is the characteristic function of a compact set $K$, and the value $S(\chi_K\varphi)$ is understood as the weighted integral over the support of $S$. More generally, a current $S\in\mathcal R_n^{\mathit{loc}}$ is called {\em stationary} if for all compact subsets $K\subset M$ the relation
$$
\left.\frac{d}{dt}\right|_{t=0}\mathbf{M}((\phi_t^V)_*(S\llcorner K))=0
$$
holds for all vector fields $V$ with support in $K$, where $\phi_t^V$ is the flow of $V$.

In the sequel we often consider $M$ to be the hyperbolic space $\mathbf H^{n+1}$, that is the complete simply connected $(n+1)$-dimensional Riemannian manifold whose all sectional curvatures are equal to $-1$. It is convenient to identify $\mathbf H^{n+1}$ with an open unit ball $\mathbb{B}^{n+1}$ via the Poincar\'e model. Every point on the boundary sphere $p\in\partial\mathbb{B}^{n+1}$ represents an equivalence class of asymptotic geodesics in $\mathbf{H}^{n+1}$, and $\partial\mathbb{B}^{n+1}$ is naturally identified with the so-called boundary at infinity, also denoted by $S^n_\infty(\mathbf H^{n+1})$. The following statement by Anderson~\cite{An82} is fundamental for our main result, Theorem~\ref{t1}.
\begin{prop}
\label{p:a}
Let $\Gamma^{n-1}\subset S^n_\infty(\mathbf H^{n+1})$ be a closed oriented submanifold in the sphere at infinity, where $n\geqslant 2$. Then there exists a complete area-minimising locally integral $n$-current $\Sigma^n$ in $\mathbf H^{n+1}$ that is asymptotic to $\Gamma^{n-1}$ in the sense that $\overline{\spt \Sigma^n}\cap S^n_\infty(\mathbf H^{n+1})=\Gamma^{n-1}$. 
\end{prop}

The existence of area-minimising currents continues to hold in higher codimension and under weaker assumptions on the submanifold $\Gamma\subset S^n_\infty(\mathbf H^{n+1})$, see~\cite{An82}. In particular, in the sequel we may always assume that $\Gamma$ is only $C^{1,\alpha}$-smooth, where $0\leqslant\alpha\leqslant 1$. By standard regularity theory, see~\cite{FM,LS}, the area-minimising current in Proposition~\ref{p:a} is a smooth embedded manifold in the complement of the singular set of Hausdorff dimension at most $n-7$. For the sequel we need the following important boundary regularity statement due to Hardt and Lin~\cite{HL87}.

\begin{prop}
\label{p:hl}
Let $\Gamma^{n-1}\subset S^n_\infty(\mathbf H^{n+1})$ be a closed $C^{1,\alpha}$-smooth submanifold in the sphere at infinity, where $0\leqslant\alpha\leqslant 1$, and let $\Sigma^n$ be an area-minimising locally rectifiable $n$-current in $\mathbf H^{n+1}$ that is asymptotic to $\Gamma^{n-1}$. Then there exists a neighbourhood $U$ of $\Gamma^{n-1}$ in the closure $\bar{\mathbf{H}}^{n+1}$ such that there is no singular set of $\Sigma^n$ in $U$, and the union $(\spt\Sigma^n\cap U)\cup\Gamma^{n-1}$ is a $C^{1,\alpha}$-smooth manifold with boundary that meets the ideal boundary $S^n_\infty(\mathbf H^{n+1})$ orthogonally. 
\end{prop}

In~\cite{L89a} Lin also shows that if $\Gamma^{n-1}\subset S^n_\infty(\mathbf H^{n+1})$ in Proposition~\ref{p:hl} is $C^{k,\alpha}$-smooth, where $1\leqslant k\leqslant n-1$  and $0\leqslant\alpha\leqslant 1$ or $k=n$ and $0\leqslant\alpha<1$, then $(\Sigma^n\cap U)\cup\Gamma^{n-1}$ is a $C^{k,\alpha}$-smooth manifold with boundary. The boundary regularity in higher codimension is a more subtle question. However, the following result, due to~\cite{L89b}, guarantees the existence of an area-mininising current regular near the boundary.

\begin{prop}
\label{p:l}
Let $\Gamma^{m-1}\subset S^n_\infty(\mathbf H^{n+1})$ be a closed $C^{1,\alpha}$-smooth submanifold in the sphere at infinity, where $0\leqslant\alpha\leqslant 1$. Then there exists a complete area-minimising locally rectifiable $m$-current $\Sigma^m$ in $\mathbf H^{n+1}$ that is asymptotic to $\Gamma^{m-1}$ at infinity. Moreover, near $\Gamma^{m-1}$ the set $\spt\Sigma^m\cup\Gamma^{m-1}$ is a $C^{1,\alpha}$-smooth submanifold with boundary that meets the ideal boundary $S^n_\infty(\mathbf H^{n+1})$ orthogonally.
\end{prop}

Thus, for area-minimising currents in Proposition~\ref{p:l} the singular set also lies in a bounded subset in $\mathbf{H}^{n+1}$, and in the sequel we shall use this fact essentially. Note also that in higher codimension $n-m>0$ the standard regularity theory guarantees only that the Hausdorff dimension of the singular set is at most $m-2$.

\subsection{Laplacian and its spectrum}
Now we describe analytic background related to the notion of the spectrum of the Laplace operator on area-minimising currents in Propositions~\ref{p:a} and~\ref{p:l}.

Let $(\Sigma_*,g)$ be a Riemannian manifold, and let $\Delta$ be the Laplace-Beltrami operator on $\Sigma_*$ with the sign convention so that it is non-negative. We view $\Delta$ as the operator defined on the subspace $\mathcal D(\Delta)\subset L^2(\Sigma_*)$  formed by compactly supported smooth functions on $\Sigma_*$. By standard Green's formula
$$
\int_{\Sigma_*}(-\Delta u)vd\mathit{Vol}_g+\int_{\Sigma_*}\langle\nabla u,\nabla v\rangle d\mathit{Vol}_g=0,
$$
where $u$, $v\in\mathcal D(\Delta)$, it is symmetric and positive-definite. Since it is also real, it is straighforward to show that it admits a self-adjoint extension to an unbounded operator in $L^2(\Sigma_*)$, see for example~\cite[Lemma~1.2.8]{BD}. Depending on the geometry of $\Sigma_*$, there can be many self-adjoint extensions of $\Delta$. For example, when the boundary of $\Sigma_*$ is non-empty, different boundary conditions correspond to different self-adjoint extensions, see~\cite{BD,MT}. On the other hand, when $\Sigma_*$ is complete, or is a regular locus of irreducible analytic subvariety, the Laplace operator $\Delta$ above is known to have a unique self-adjoint extension, see~\cite{Ko20,LT95,MT}. Such operators are called essentially self-adjoint.

Now let $\Sigma^n$ be a stationary locally rectifiable current in an oriented complete Riemannian manifold $M$. Below as $\Sigma_*$ we take the {\em regular set} $\reg\Sigma$, that is the subset of $\spt\Sigma^n$ formed by points $x$ such that $B_\varepsilon(x)\cap\spt\Sigma^n$ is a smooth embedded submanifold for some $\varepsilon>0$. To our knowledge it is unknown whether the Laplacian described above is essentially self-adjoint on $\reg\Sigma^n$. To avoid any ambiguity we consider the spectrum of the Friedrichs extension $\bar\Delta$ to a self-adjoint unbounded linear operator. This extension is obtained as the self-adjoint operator associated with the closure of the Dirichlet energy, defined on $W_0^{1,2}(\Sigma^n_*)$. We refer to~\cite[Section~4.4]{BD} for related material on operators associated with quadratic forms.

By $\lambda_0(\Sigma^n)$ we denote the quantity
\begin{equation}
\label{lam:def1}
\lambda_0(\Sigma^n)=\inf\left(\int \abs{\nabla u}^2d\norm{\Sigma}\right)/\left(\int u^2d\norm{\Sigma}\right),
\end{equation}
where the infimum is taken over non-trivial smooth functions with compact support in $\reg\Sigma^n$. For the Friedrichs extension $\bar\Delta$ of the Laplacian it is straightforward to show that the operator $\bar\Delta-\lambda_0(\Sigma^n)$ is non-negative, and by standard theory, see for example~\cite[Theorem~4.3.1]{BD}, the value $\lambda_0(\Sigma^n)$ is precisely the bottom of the spectrum of $\bar\Delta$. Equivalently, the value $\lambda_0(\Sigma^n)$ can be defined as the infimum of the zero Dirichlet eigenvalues $\lambda_0(\Omega)$, where $\Omega$ ranges over all open submanifolds $\Omega\subset\reg\Sigma^n$ such that the closure $\bar\Omega$ is compact and the boundary $\partial\Omega$ is smooth.

Finally, note that the version of the quantity $\lambda_0(\Sigma^n)$ can be defined for all currents $T\in\mathcal R_n^{\mathit{loc}}$ by the same relation. In more detail, denote by $\reg_1 T$ the subset of $\spt\Sigma^n$ formed by points $x$ such that $B_\varepsilon(x)\cap\spt\Sigma^n$ is a $C^1$-smooth embedded submanifold for some $\varepsilon>0$. Note that the $n$-dimensional Hausdorff measure of the complement $\spt T\backslash\reg_1 T$ is zero, see~\cite{FM}. Allowing the function $u$ in formula~\eqref{lam:def1} to range among non-trivial $C^1$-smooth functions with compact support in $\reg_1 T$, we arrive at the quantity $\lambda_*(T)$, called the {\em fundamental tone}.

\begin{remark}
It is worth mentioning that for a stationary current $\Sigma^n$ the sets $\reg_1\Sigma^n$ and $\reg\Sigma^n$ coincide. This statement is a consequence of Allard's regularity theorem together with elliptic regularity, see~\cite{GM12,LS}. For area-minimising currents it can be deduced independently from the observation that for any $x\in\reg_1\Sigma^n$ the oriented tangent cone is an oriented $n$-dimensional plane,  see~\cite[Lemma~10.3]{FM}. As a consequence, we conclude that for a stationary current $\Sigma^n$ the values $\lambda_0(\Sigma^n)$ and $\lambda_*(\Sigma^n)$ coincide. Note also that, by a standard argument based on the first variation, the regular set $\reg\Sigma^n$ can be viewed as a genuine minimal (not necessarily proper) submanifold, that is a submanifold of zero mean curvature.
\end{remark}

\begin{remark}
In the spirit of~\cite{Ko14}, see also~\cite[Section~4.5]{BD}, using the Dirichlet form, one can also define higher variational eigenvalues $\lambda_k(T)$ for general currents. However, in the context of the main results of the current paper, their meaning does not add anything new. In more detail, if $\Sigma^n$ is an area-minimising locally rectifiable $n$-current that satisfies the hypotheses of Theorem~\ref{t1} or Theorem~\ref{t2}, then the argument in the proof of Theorem~\ref{t2} shows that the value $(n-1)^2/4$ lies in the essential spectrum of the Laplacian. Since this value is also the bottom of the spectrum, by standard results, see for example~\cite[Theorem~4.5.2]{BD}, all higher variational eigenvalues of $\Sigma^n$ are equal to $(n-1)^2/4$.
\end{remark}

\section{Comparison of the isoperimetric constants}
\label{comp}

\subsection{Main statement}
As above, let $M$ be an oriented complete Riemannian manifold, and let $\Sigma^m$ be a locally rectifiable $m$-current in $M$, $\Sigma^m\in\mathcal{R}_m^{\mathit{loc}}$. For any compact subset $K\subset M$ by the {\em isoperimetric constant} $h(\Sigma^m\llcorner K)$ we call the quantity
$$
h(\Sigma^m\llcorner K)=\inf\{\mathbf{M}(\partial\Omega)/\mathbf{M}(\Omega): \bar\Omega\subset\reg\Sigma^m\cap K\},
$$
where $\Omega$ ranges over all open submanifolds of $\reg\Sigma^m\cap K$ whose closure is compact, and whose boundary is smooth. If there is no such a submanifold $\Omega$, for example when $\reg\Sigma^m\cap K$ is empty, then we set $h(\Sigma^m\llcorner K)$ to be $+\infty$. By $h(\mathbf{B}_r^m)$ we denote the isoperimetric constant of the ball of radius $r>0$ in the $m$-dimensional hyperbolic space $\mathbf{H}^m$, that is
$$
h(\mathbf{B}_r^m)=\inf\{\mathit{Area}(\partial\Omega)/\mathit{Vol}(\Omega): \bar\Omega\subset\mathbf{B}_r^m\}.
$$
As is known, see also Remark~\ref{iso:rem} below, the value $h(\mathbf{B}_r^m)$ is given by the formula
\begin{equation}
\label{iso:h}
h(\mathbf{B}_r^m)=\sinh^{m-1}(r)/\left(\int_0^r\sinh^{m-1}(t)dt\right).
\end{equation}
Our first result here is the following comparison theorem.
\begin{theorem}
\label{tc}
Let $M$ be a complete Riemannian manifold whose sectional curvatures are at most $-1$, and let $\Sigma^m$ be a stationary locally rectifiable $m$-current in $M$. Then for any geodesic ball $B_r(p)\subset M$ the following inequality holds 
$$
h(\Sigma^m\llcorner \bar B_r(p))\geqslant h(\mathbf{B}_r^m)
$$
for all $0<r\leqslant \inj_p(M)$, where $\inj_p(M)$ is the injectivity radius of $M$ at a point $p$. 
\end{theorem}
It is straightforward to show, see for example~\cite[Lemma~2.5]{Ko21}, that
\begin{equation}
\label{sinh:ineq}
h(\mathbf{B}_r^m)\geqslant (m-1)\coth(r)>(m-1).
\end{equation}
If  there exists a point $p\in M$ such that $\inj_p(M)=+\infty$, which occurs, for example, when $M$ is simply connected, Theorem~\ref{tc} immediately implies the version of Yau's inequality for all domains $\Omega\subset\reg\Sigma^m$ with smooth boundary and compact closure:
$$
\mathbf{M}(\partial\Omega)>(m-1)\mathbf{M}(\Omega).
$$
Further, using Cheeger's inequality, we obtain
$$
\lambda_*(\reg\Sigma^m\cap B_r(p))\geqslant \frac{1}{4}h^2(\mathbf{B}_r^m),
$$
and again when $\inj_p(M)=+\infty$, the combination with~\eqref{sinh:ineq} yields the following version of McKean's inequality~\cite{McK70} for the fundamental tone of a stationary locally rectifiable current $\Sigma^m$:
\begin{equation}
\label{MK}
\lambda_*(\Sigma^m)\geqslant\frac{1}{4}(m-1)^2.
\end{equation}
For smooth minimal submanifolds this version of McKean's inequality can be derived from~\cite[Theorem~1.10]{BM07}, which is an improvement of earlier results in~\cite{CLY,Ma86}.

\subsection{Proof of Theorem~\ref{tc}}
We essentially follow the comparison argument in~\cite{Ko21,LMMV}. Although the set $\spt\Sigma^m\cap B_r(p)$ may contain singular points, the argument carries over due to monotonicity of certain quantities involved. We explain this below in more detail.

Let us denote by $r(x)$ the distance function $\dist(p,x)$ in $M$ restricted to the regular set $\reg\Sigma^m$. Since the latter can be viewed as a (non-complete) submanifold with vanishing mean curvature, the comparison argument in the proof of~\cite[Lemma~2.7]{Ko21} yields the inequality
\begin{equation}
\label{comparison}
-\Delta r(x)\geqslant\coth(r(x))(m-\abs{\nabla r(x)}^2)
\end{equation}
for any $x\in\reg\Sigma^m$ such that $0<r(x)<\inj_p(M)$. Now consider the function
$$
f(r)=\int\limits_0^r \frac{dt}{h_t},\qquad\text{where~}\quad r>0,
$$
and $h_t$ is the function defined by the right-hand side in~\eqref{iso:h}. A straightforward computation shows  that $f(r)$ is convex, see~\cite[Corollary~2.6]{Ko21}, and satisfies the relations
\begin{equation}
\label{f:rels}
f''(r)+(m-1)\coth(r)f'(r)=1,\qquad f(0)=0,\quad f'(0)=0.
\end{equation}
Denote by $\psi$ the function on $B_r(p)\cap\reg\Sigma^m$ defined as the composition $\psi(x)=f\circ r(x)$, where $r(x)=\dist(p,x)$. Computing the Laplacian of $\psi$, we obtain
\begin{multline*}
-\Delta\psi=f''(r)\abs{\nabla r}^2-f'(r)\Delta r\geqslant f''(r)\abs{\nabla r}^2+f'(r)\coth(r)(m-\abs{\nabla r}^2)\\
=1+(1-\abs{\nabla r}^2)\left(f'(r)\coth(r)-f''(r)\right),
\end{multline*}
where we used relations~\eqref{comparison}-\eqref{f:rels}. It is straightforward to check that the term
$$
f'(r)\coth(r)-f''(r)=m\coth(r)f'(r)-1
$$
is non-negative, see~\cite[Lemma~2.5]{Ko21}, and we conclude that $-\Delta\psi\geqslant 1$. Now let $\Omega$ be an open submanifold such that the closure $\bar\Omega$ lies in $B_r(p)\cap\reg\Sigma^m$ and the boundary $\partial\Omega$ is smooth. Then, using the divergence theorem, we obtain
$$
\mathbf{M}(\Omega)\leqslant-\int_\Omega \Delta\psi d\norm{\Sigma^m}=\int_{\partial\Omega}\langle\grad\psi,\nu\rangle\leqslant f'(r)\mathbf{M}(\partial\Omega),
$$
where $\nu$ is a unit normal vector, and we used that $f'(r)=1/h_r$ is increasing. Re-arranging, we immediately obtain 
\begin{equation}
\label{in:f}
\sinh^{m-1}(r)/\left(\int_0^r\sinh^{m-1}(t)dt\right)\leqslant \mathbf{M}(\partial\Omega)/\mathbf{M}(\Omega),
\end{equation}
and since $\Omega$ is arbitrary, finish the proof of the theorem.
\qed
\begin{remark}
\label{iso:rem}
Note that the quotient 
$$
\sinh^{m-1}(r)/\left(\int_0^r\sinh^{m-1}(t)dt\right)
$$
is precisely the ratio $\mathit{Area(\partial\mathbf{B}^m_r)}/\mathit{Vol}(\mathbf{B}^m_r)$, where $\mathbf{B}^m_r$ is the ball of radius $r$ in the $m$-dimensional hyperbolic space $\mathbf{H}^{m}$. Thus, applying the argument in the proof of Theorem~\ref{tc} to the totally geodesic subspace $\mathbf{H}^m\subset\mathbf{H}^{n+1}$ in place of $\Sigma^m$, relation~\eqref{in:f} shows that the isoperimetric constant $h(\mathbf{B}^m_r)$ is given by formula~\eqref{iso:h}. This formula can be also deduced directly from the classical isoperimetric inequality in the hyperbolic space.
\end{remark}
\begin{remark}
\label{rem:cones}
For a fixed point $p\in M$ consider an arbitrary geodesic cone $C^m$ with the origin at $p$. For example, if $M$ is the hyperbolic space $\mathbf{H}^{n+1}$, and $\Gamma^{m-1}$ is a submanifold in the sphere at infinity, then as $C^m$ one can take the cone formed by geodesic rays emanating from $p$ and asymptotic to points in $\Gamma^{m-1}$. Let $r(x)=\dist(x,p)$ be the corresponding distance function, where $x\in C^m$. Inspecting the proof of~\cite[Lemma~2.7]{Ko21}, we see that relation~\eqref{comparison} continues to hold; this is a consequence of the fact that the gradient $\grad_x r$, where the function $r$ is viewed as a function on $M$, lies in the tangent space $T_xC^m$. Then the argument in the proof of Theorem~\ref{tc} shows that
$$
h(C^m\llcorner \bar B_r(p))\geqslant h(\mathbf{B}_r^m),
$$
where $p$ is the origin of $C^m$. In particular, we conclude that McKean's inequality~\eqref{MK} holds for arbitrary, not necessarily minimal, cone.
\end{remark}
\begin{remark}
The statement and the proof of Theorem~\ref{tc} carry over directly to the setting when $M$ has non-positive sectional curvatures. In this case the comparison inequality for isoperimetric constants takes the form
$$
h(\Sigma^m\llcorner \bar B_r(p))\geqslant h(\mathbb{B}_r^m),
$$
where $\mathbb{B}_r^m$ is an $m$-dimensional Euclidean ball of radius $r>0$. In fact, if the point $p\in M$ in these inequalities is fixed, then the curvature hypotheses above, as well as in Theorem~\ref{tc} can be weakened -- it is sufficient to impose the bound only on the sectional curvatures along two-dimensional subspaces containing the radial vector $\grad_xr$, where $r(x)=\dist(p,x)$. 
\end{remark}

\section{Proofs}
\label{proofs}

\subsection{Proof of Theorem~\ref{t2}: the inequality}
Let $T^{m}\subset\mathbf{H}^{n+1}$ be a locally rectifiable $m$-current from the class $\mathcal M(\Gamma^{m-1})$ defined in Section~\ref{intro}. We start with a proof of the inequality
\begin{equation}
\label{blam}
\lambda_*(T^m)\leqslant\frac{1}{4}(m-1)^2.
\end{equation}
By the variational characterisation, see the discussion in Section~\ref{prems}, for this it is sufficient to construct a sequence of compactly supported Lipschitz functions $\phi_k$ on $\reg_1 T^m$ such that the Rayleigh quotients satisfy the relation
$$
\lim\sup_{k}\mathcal R(\phi_k)\leqslant\frac{1}{4}(m-1)^2,
$$
where
$$
\mathcal R(\phi_k)=\left(\int\abs{\nabla\phi_k}^2d\norm{T}\right)/\left(\int\phi_k^2d\norm{T}\right).
$$
The idea is to construct $(\phi_k)$ as a sequence of radial functions on the ambient space $\mathbf{H}^{n+1}$ whose restriction to $\spt T^m$ satisfies the desired relation above. To make the exposition more explicit we 
consider first the case of cones in $\mathcal M(\Gamma^{m-1})$, where our argument is very close the one contained in~\cite{Cha}, and then handle the case of arbitrary currents  $T^m\in\mathcal M(\Gamma^{m-1})$, using integral estimates near $\Gamma^{m-1}$.

\medskip
\noindent
{\em Step~1: radial test-functions and cones.} Following~\cite{Cha}, we define the following family of functions
$$
u_R(t)=\left\{
\begin{array}{ll}
\exp\left({-{\frac{(m-1)}{2}}t}\right)\sin\left({\frac{2\pi}{R}}\left(t-\frac{R}{2}\right)\right), & \text{ if } t\in[R/2,R];\\
0, & \text{otherwise}.
\end{array}
\right.
$$
It is straightforward to check that it satisfies the equation
$$
u_R''+(m-1)u_R'+\left(\frac{1}{4}(m-1)^2+\frac{4\pi^2}{R^2}\right)u_R=0
$$
on $(R/2,R)$. Integrating by parts, and using the relation above, we obtain
\begin{multline}
\label{aux1}
\int\limits_{R/2}^R(u'_R(t))^2\sinh^{m-1}(t)dt\\ =\int\limits_{R/2}^R\left(-(m-1)u_R(t)u_R'(t)(\coth(t)-1)+\left(\frac{1}{4}(m-1)^2+\frac{4\pi^2}{R^2}\right)u_R^2(t)\right)\sinh^{m-1}(t)dt.
\end{multline}
Now fix a point $p\in\mathbf{H}^{n+1}$, for example the origin of the corresponding unit ball $\mathbb B^{n+1}$, and denote by $r(x)$ the hyperbolic distance $\dist(p,x)$. Then the function $u_R\circ r$ is a compactly supported Lipschitz function on $\mathbf{H}^{n+1}$. Consider the cone obtained as the union of all geodesic rays emanating from $p$ and asymptotic to a point in $\Gamma^{m-1}$. For example, if $p$ is the origin of the unit ball $\mathbb B^{n+1}$, then this cone is the subset
\begin{equation}
\label{cone:def}
C^m=\{\tau z: z\in\Gamma^{m-1}, \tau\in(0,1)\}\subset\mathbb B^{n+1},
\end{equation}
where we identify $\mathbf{H}^{n+1}$ with the unit ball $\mathbb B^{n+1}$ via the Poincar\'e model.
Let $\phi_R$ be the restriction of the function $u_R\circ r$ to the cone $C^m$. Then, using relation~\eqref{aux1}, in geodesic spherical coordinates we obtain
\begin{multline*}
\int\abs{\nabla\phi_R}^2d\norm{C^m}-\left(\frac{1}{4}(m-1)^2+\frac{4\pi^2}{R^2}\right)\int{\phi_R}^2d\norm{C^m}=\\ \omega(\Gamma^{m-1})\left(\int\limits_{R/2}^R(u'_R(t))^2\sinh^{m-1}(t)dt - \left(\frac{1}{4}(m-1)^2+\frac{4\pi^2}{R^2}\right)\int\limits_{R/2}^R u_R^2(t)\sinh^{m-1}(t)dt\right) \\ \leqslant \omega(\Gamma^{m-1})\int\limits_{R/2}^R (m-1)\abs{\coth(t)-1}\abs{u_R(t)}\abs{u'_R(t)}\sinh^{m-1}(t)dt \\ \leqslant(m-1)\left(\coth\left(\frac{R}{2}\right)-1\right)\left(\int\abs{\nabla\phi_R}^2d\norm{C^m}\right)^{1/2}\left(\int\abs{\phi_R}^2d\norm{C^m}\right)^{1/2},
\end{multline*}
where $\omega(\Gamma^{m-1})$ is the volume of $\Gamma^{m-1}$ viewed as a submanifold of the unit sphere $\partial\mathbb B^{n+1}$. Thus, denoting by $\mathcal R(\phi_R)$ the Rayleigh quotient of $\phi_R$, we get
$$
\mathcal R(\phi_R)-A_R\leqslant B_R\sqrt{\mathcal R(\phi_R)}, 
$$
where
$$
A_R=\frac{1}{4}(m-1)^2+\frac{4\pi^2}{R^2},\qquad B_R=\coth\left(\frac{R}{2}\right)-1.
$$
By elementary analysis the last inequality gives
\begin{equation}
\label{RIN}
\sqrt{\mathcal R(\phi_R)}\leqslant \frac{B_R}{2}+\left(A_R+\frac{B_R^2}{4}\right)^{1/2},
\end{equation}
and tending $R\to+\infty$, we arrive at inequality~\eqref{blam}.

\medskip
\noindent
{\em Step~2: general case.} Let $p\in\mathbf{H}^{n+1}$ be a point that corresponds to the origin in the unit ball $\mathbb B^{n+1}$ via the Poincar\'e model. Below by the {\em radial function} on $\mathbf{H}^{n+1}$ we mean a function that depends on the distance $r(x)=\dist(x,p)$ only. We proceed with the following lemma.
\begin{lemma}
\label{vol}
Let $T^m$ be a locally rectifiable $m$-current from the set $\mathcal M(\Gamma^{m-1})$, and let $C^m$ be the cone given by~\eqref{cone:def}. Then, for any $0<\varepsilon<1$ there exists $R>0$ such that for any non-negative bounded radial function $f$ with compact support in the complement of the hyperbolic ball $\mathbf{B}_{R}(p)$ the inequalities 
$$
(1-\varepsilon)\!\!\!\int fd\norm{C^m}\leqslant\int fd\norm{T^m}\leqslant (1+\varepsilon)\!\!\!\int fd\norm{C^m}
$$
hold.
\end{lemma}
\begin{proof}
The hypotheses in the definition of the set $\mathcal M(\Gamma^{m-1})$, see Section~\ref{intro}, guarantee that for any point $z\in\Gamma^{m-1}$ there exists its neighbourhood $V_z$ in the closed unit ball $\overline{\mathbb B^{n+1}}$ where  $\spt T^m$ is the graph of a $C^1$-smooth vector-function $w$. In more detail, consider the tangent space
$$
T_zT^m=\mathbb{R}\times T_z\Gamma^{m-1},
$$
where the real line component is spanned by the radial vector $\partial/\partial r$. The latter indeed lies in the tangent space $T_zT^m$, since $\spt T^m$ meets the boundary at infinity orthogonally. A standard argument based on the inverse function theorem shows that there exist a neighbourhood $(-\delta,0]\times D_\sigma$ of the origin in the half-space $(\mathbb{R}_{-})\times T_z\Gamma^{m-1}$ and a vector-function $w$ such that in the spherical coordinates near $z$ we have
\begin{equation}
\label{graph}
V_z\cap\spt T^m=\{(t,\theta,w(t,\theta)): t\in (1-\delta,1], \theta\in D_\sigma\},
\end{equation}
where $t=\abs{x}$ is the Euclidean radial distance, and we identify $T_z\Gamma^{m-1}$ with the coordinate subspace $\theta$ in local coordinates on the unit sphere $\partial\mathbb B^{n+1}$. Note that, since $\nabla w$ vanishes at the origin in $T_zT^m$, the Euclidean norm $\abs{\nabla w}_\infty$ can be made arbitrarily small by choosing sufficiently small $\delta$ and $\sigma$. Choosing coordinates $(\theta^1,\ldots,\theta^{m-1})$ on $D_\sigma$ such that the vectors $\partial/\partial\theta^i$ are orthonormal at the origin,  it is straightforward to see that the integral of a radial function $f$ in the chart defined by the graph in~\eqref{graph} is given by the formula
$$
\int\limits_{V_z\cap\spt T^m}fd\norm{T^m}=\int\limits_{1-\delta}^{1}\int\limits_{D_\sigma}f(t)\frac{2^m}{(1-t^2)^m}\sqrt{1+P_{2m}(\nabla w)}dtd\theta^1\cdots d\theta^{m-1},
$$
where $P_{2m}(\nabla w)$ is a polynomial of degree $2m$ in derivatives of $w$ that vanishes at the point $z$. Since the tangent spaces $T_zT^m$ and $T_zC^m$ coincide, choosing $V_z$ smaller, if necessary, we have a similar formula for the cone $C^m$:
$$
\int\limits_{V_z\cap\spt C^m}fd\norm{C^m}=\int\limits_{1-\delta}^{1}\int\limits_{D_\sigma}f(t)\frac{2^m}{(1-t^2)^m}\sqrt{1+P_{2m}(\nabla v)}dtd\theta^1\cdots d\theta^{m-1},
$$
where $v$ is a $C^1$-smooth vector-function defined on the same set $(1-\delta,1]\times D_\sigma$. Now for a given $0<\varepsilon<1$ we may choose $\delta$ and $\sigma$ such that
$$
(1-\varepsilon)^2\leqslant\frac{1+P_{2m}(\nabla w)}{1+P_{2m}(\nabla v)}\leqslant (1+\varepsilon)^2
$$
Combining these inequalities with the formulae for the integrals above, we immediately obtain
$$
(1-\varepsilon)\!\!\!\int\limits_{V_z\cap\spt C^m}fd\norm{C^m}\leqslant\int\limits_{V_z\cap\spt T^m}fd\norm{T^m}\leqslant (1+\varepsilon)\!\!\!\int\limits_{V_z\cap\spt C^m}fd\norm{C^m}.
$$
Our next aim is to produce similar inequalities in a neighbourhood of $\Gamma^{m-1}$, using an appropriate version of the partition of unity. For the completeness of exposition, we explain this construction below.

Since $\Gamma^{m-1}$ is compact, we may find a finite collection of points $z_i$, where $i=1,\ldots,\ell$,  such that the sets
$$
W_{z_i}=\{(1,\theta,w_i(1,\theta)): \theta\in D_{\sigma_i}\}=\{(1,\theta,v_i(1,\theta)): \theta\in D_{\sigma_i}\},
$$
cover $\Gamma^{m-1}$, where the vector-functions $w_i$ and $v_i$ are constructed as described above. Then the corresponding neighbourhoods $\{V_{z_i}\}$ cover both $\spt T^m$ and $\spt C^m$ near $\Gamma^{m-1}$. In particular, there exists $R>0$ such that the complements
$$
\spt T^m\backslash \mathbf{B}_R(p)\qquad\text{and}\qquad \spt C^m\backslash\mathbf{B}_R(p)
$$
lie in the union of the $V_{z_i}$'s. Now let $\{\varphi_i\}$ be a partition of unity on $\Gamma^{m-1}$ subordinate to the covering $\{W_{z_i}\}$. Since the charts $V_{z_i}\cap T^m$ and $V_{z_i}\cap C^m$ are parameterised by the same coordinates $(t,\theta)$ we may also define functions $\eta_i$ and $\zeta_i$ on these sets respectively, by setting
$$
\eta_i(t,\theta)=\varphi_i(\theta)\qquad\text{and}\qquad\zeta_i(t,\theta)=\varphi_i(\theta).
$$
Then the discussion above shows that for any non-negative radial function, supported in the complement of the hyperbolic ball $\mathbf{B}_R(p)$, the following relations hold:
\begin{equation}
\label{ith}
(1-\varepsilon)\!\!\!\int\zeta_i fd\norm{C^m}\leqslant\int\eta_i fd\norm{T^m}\leqslant (1+\varepsilon)\!\!\!\int\zeta_{i} f d\norm{C^m},
\end{equation}
where we view $\eta_if$ and $\zeta_if$ as compactly supported functions on $\spt T^m$ and $\spt C^m$ respectively, and $i=1,\ldots,\ell$. Since $\{\varphi_i\}$ is a partition of unity on $\Gamma^{m-1}$, by our construction we immediately obtain the relations
$$
\sum_i\eta_i=1\quad\text{~on~}\spt T^m\backslash\mathbf{B}_R(p),\qquad \sum_i\zeta_i=1\quad\text{~on~}\spt C^m\backslash\mathbf{B}_R(p),
$$
and summing inequalities in~\eqref{ith}, arrive at the statement of the lemma.
\end{proof}

We proceed with a proof of relation~\eqref{blam}. For a given $0<\varepsilon<1$ let $R_*>0$ be a real number that satisfies the conclusions of Lemma~\ref{vol}. By the definition of the class $\mathcal M(\Gamma^{m-1})$ we may also assume that $\spt T^m$ consists of regular points in the complement of $\mathbf{B}_{R_*}(p)$. Choose $R>2R_*$ and denote by $\tilde\phi_R$ the restriction to $\spt T^m$ of the function $u_R$, defined above. It is straightforward to see that
$$
\abs{\nabla\tilde\phi_R(r(x))}\leqslant\abs{u_R'(r(x))}=\abs{\nabla\phi_R(r(\bar x))},
$$
where $x\in\spt T^m$ and $\bar x\in C^m$ are such that $r(x)=r(\bar x)$, and by Lemma~\ref{vol}, we obtain
$$
\mathcal R(\tilde\phi_R)\leqslant\frac{1+\varepsilon}{1-\varepsilon}\mathcal R(\phi_R).
$$
Combining the latter with relation~\eqref{RIN}, and tending $R\to +\infty$, we immediately arrive at
$$
\lambda_*(T^m)\leqslant\frac{1}{4}\frac{(1+\varepsilon)}{(1-\varepsilon)}{(m-1)^2}.
$$
Since $0<\varepsilon<1$ is arbitrary, we are done.
\qed

\subsection{Proof of Theorem~\ref{t2}: the spectrum of the area-minimising current}
By Proposition~\ref{p:l} there exists an area-minimising locally rectifiable $m$-current $\Sigma^m$ that lies in the class $\mathcal M(\Gamma^{m-1})$. By McKean's inequality~\eqref{MK}, we conclude that upper bound~\eqref{blam} for the bottom of the spectrum is saturated on $\Sigma^m$. Thus, for a proof of the theorem it remains to show that any $\lambda>(m-1)^2/4$ belongs to the spectrum of $\Sigma^m$. By standard theory, see for example~\cite[Lemma~4.1.2]{BD}, for the latter it is sufficient to construct a sequence $(\phi_k)$ of smooth compactly supported functions on $\reg\Sigma^m$ and a sequence of positive real numbers $(\epsilon_k)$, $\epsilon_k\to 0+$, such that
\begin{equation}
\label{r:aim}
\int\abs{\Delta\phi_k-\lambda\phi_k}^2d\norm{\Sigma^m}\leqslant\epsilon_k\int\abs{\phi_k}^2d\norm{\Sigma^m}\qquad\text{for all~} k\in\mathbb N.
\end{equation}
In fact, we shall construct a sequence of functions with mutually disjoint supports, showing that the interval $[(m-1)^2/4,+\infty)$ is precisely the essential spectrum of $\Sigma^m$.

As above, we first construct radial test-functions on the the ambient space $\mathbf{H}^{n+1}$, and show how to use them to compute the spectrum of cones. Then, we perform a version of this argument for the regular set of the area-minimising current $\Sigma^m$.

\medskip
\noindent
{\em Step~1:radial test-functions and cones.} We start with introducing auxuliary functions modelled on the ones used to study the spectrum of the hyperbolic space. First, for a given $\lambda>(m-1)^2/4$ we define the complex-valued function
$$
\psi(t)=\sinh^{-(m-1)/2}(t)\exp(i\beta t),\qquad\text{where~}\beta=\sqrt{\lambda-\frac{1}{4}(m-1)^2}.
$$
A direct calculation shows that it satisfies the relation
\begin{equation}
\label{psi}
\psi''(t)+(m-1)\coth(t)\psi'(t)+(\lambda+\alpha(t))\psi(t)=0,
\end{equation}
where we set
$$
\alpha(t)=\frac{1}{4}(m-1)(m-3)\sinh^{-2}(t).
$$
Further, for a given $R>0$ we consider the function
$$
\upsilon_R(t)=\left\{
\begin{array}{ll}
\psi(t)\sin^2\left(\frac{2\pi}{R}\left(t-\frac{R}{2}\right)\right), & \text{if~ } t\in[R/2,R];\\
0, & \text{otherwise}.
\end{array}
\right.
$$
Note that the second derivative $\upsilon_R''$ is a bounded function with compact support. The next lemma shows that $\upsilon_R$ is indeed a good model function.
\begin{lemma}
\label{est}
For any $R>0$ the function $\upsilon_R$, defined above, satisfies the inequality
$$
\int\limits_{R/2}^{R}\abs{\upsilon_{R}''(t)+(m-1)\coth(t)\upsilon_{R}'(t)+\lambda\upsilon_{R}(t)}^2\sinh^{m-1}(t)dt\leqslant \epsilon_R\int\limits_{R/2}^{R}\abs{\upsilon_R(t)}^2\sinh^{m-1}(t)dt,
$$
where 
$$
\epsilon_R=2\max\left\{\alpha^2\left(\frac{R}{2}\right), 16\abs{\beta}^2\left(\frac{2\pi}{R}\right)^2, 16\left(\frac{2\pi}{R}\right)^4\right\}.
$$
In addition, for a fixed $R_0>0$ and any $R>R_0$ the following inequalities hold:
\begin{equation}
\label{est:in1}
\int\limits_{R/2}^{R}\abs{\upsilon'_R(t)}^2\sinh^{m-1}(t)dt\leqslant C_*\int\limits_{R/2}^{R}\abs{\upsilon_R(t)}^2\sinh^{m-1}(t)dt,
\end{equation}
\begin{equation}
\label{est:in2}
\int\limits_{R/2}^{R}\abs{\upsilon''_R(t)}^2\sinh^{m-1}(t)dt\leqslant C_{*}\int\limits_{R/2}^{R}\abs{\upsilon_R(t)}^2\sinh^{m-1}(t)dt,
\end{equation}
where the constant $C_*$ depends on $m$, $\lambda$, and $R_0$.
\end{lemma}
\begin{proof}
Using~\eqref{psi}, a direct calculation gives
\begin{multline}
\label{up1}
\upsilon_R''(t)+(m-1)\coth(t)\upsilon_R'(t)+\lambda\upsilon_R(t)=-\alpha(t)\upsilon_R(t)+i\beta\frac{4\pi}{R}\psi(t)\sin\left(\frac{4\pi}{R}\left(t-\frac{R}{2}\right)\right) \\+2\left(\frac{2\pi}{R}\right)^2 \psi(t)\cos\left(\frac{4\pi}{R}\left(t-\frac{R}{2}\right)\right),
\end{multline}
where $t\in(R/2,R)$. Note that the product $\abs{\psi}^2\sinh^{m-1}$ equals one, and thus, we obtain
\begin{multline*}
\int\limits_{R/2}^R\abs{\psi(t)}^2\sin^2\left(\frac{4\pi}{R}\left(t-\frac{R}{2}\right)\right)\sinh^{m-1}(t)dt=\int\limits_{R/2}^R\sin^2\left(\frac{4\pi}{R}\left(t-\frac{R}{2}\right)\right)dt\\ \leqslant 4\int\limits_{R/2}^R\sin^4\left(\frac{2\pi}{R}\left(t-\frac{R}{2}\right)\right)dt=4\int\limits_{R/2}^R\abs{\upsilon_R}^2\sinh^{m-1}(t)dt,
\end{multline*}
where the inequality above follows by an elementary argument, which is omitted. Similarly, estimating the weighted $L^2$-norm of the last term in~\eqref{up1}, we get
$$
\int\limits_{R/2}^R\abs{\psi(t)}^2\cos^2\left(\frac{4\pi}{R}\left(t-\frac{R}{2}\right)\right)\sinh^{m-1}(t)dt\leqslant 4\int\limits_{R/2}^R\abs{\upsilon_R}^2\sinh^{m-1}(t)dt.
$$
Now the first inequality in the lemma follows in a straightforward manner by combining the relations above.

To prove inequality~\eqref{est:in1} one first computes
\begin{equation}
\label{psi:d}
\abs{\psi'(t)}^2\sinh^{m-1}(t)=\abs{-\frac{1}{2}(m-1)\coth(t)+i\beta}^2\leqslant C
\end{equation}
for all $t>R_0$, where the constant $C$ depends on $m$, $\lambda$, and $R_0$. Then, computing the derivative $\upsilon'_R$, we obtain
\begin{multline*}
\int\limits_{R/2}^R\abs{\upsilon'_R(t)}^2\sinh^{m-1}(t)dt\leqslant 2\int\limits_{R/2}^R\abs{\psi'(t)}^2\sin^4\left(\frac{2\pi}{R}\left(t-\frac{R}{2}\right)\right)\sinh^{m-1}(t)dt\\
+\frac{8\pi^2}{R^2}\int\limits_{R/2}^R\abs{\psi}^2\sin^2\left(\frac{4\pi}{R}\left(t-\frac{R}{2}\right)\right)\sinh^{m-1}(t)dt.
\end{multline*}
Now using inequality~\eqref{psi:d} and the observation that $\abs{\psi}^2\sinh^{m-1}$ equals one, we arrive at the inequality
\begin{multline*}
\int\limits_{R/2}^R\abs{\upsilon'_R(t)}^2\sinh^{m-1}(t)dt\leqslant 2C\int\limits_{R/2}^R\sin^4\left(\frac{2\pi}{R}\left(t-\frac{R}{2}\right)\right)dt\\
+\frac{32\pi^2}{R^2}\int\limits_{R/2}^R\sin^4\left(\frac{2\pi}{R}\left(t-\frac{R}{2}\right)\right)dt\leqslant C_1\int\limits_{R/2}^{R}\abs{\upsilon_R(t)}^2\sinh^{m-1}(t)dt,
\end{multline*}
where we used elementary inequalities between integrals of powers of $\sin$ as above. Finally, inequality~\eqref{est:in2} with some constant $C_2$ in place of $C_*$, can be derived in a similar fashion, using for example, relation~\eqref{up1}. Setting $C_*$ as the maximum of $C_1$ and $C_2$, we finish the proof of the lemma.
\end{proof}

Let $C^m$ be the cone defined by relation~\eqref{cone:def}. As above by the radial function we mean a function that depends on the distance $r(x)=\dist(x,p)$ only, where $p\in\mathbf{H}^{n+1}$ is a point that corresponds to the origin in the unit ball $\mathbb B^{n+1}$. We need the following observation, which for the convenience of references we state as a lemma. We also include a proof, where we introduce the notation that is used in the sequel.
\begin{lemma}
\label{laplace:cone}
Let $f$ be a smooth radial function on $\mathbf{H}^{n+1}$. Then the Laplacian of its restriction to the cone $C^m$ is given by the formula
$$
-\Delta f(x)=f''(r(x))+(m-1)\coth(r(x))f'(r(x)),
$$
where $r(x)=\dist(x,p)$ and $x\in C^m$.
\end{lemma}
\begin{proof}
Let $x$ be a point in the cone $C^m\subset\mathbf{H}^{n+1}$. In geodesic spherical coordinates centred at $p$ it corresponds to the value $r(x)$ and a point $\xi\in S^{n}$. Note that as $x$ ranges in $C^m$, the point $\xi$ ranges in a submanifold of $S^n$ diffeomorphic to $\Gamma^{m-1}$, which we also denote by $\Gamma^{m-1}$. In a neighbourhood of $\xi$ in $S^n$ we can choose coordinates $(\theta^1,\ldots,\theta^n)$ such that the vectors $\partial/\partial\theta^i(\xi)$ are orthonormal, and the level set
$$
\theta^m=0,\quad\ldots,\quad\theta^n=0
$$
defines $\Gamma^{m-1}$ in this neighbourhood. Then, around the point $x$ we obtain the collection of vector fields
$$
X_r=\frac{\partial}{\partial r},\qquad X_i=D\exp_p\left(\frac{\partial}{\partial\theta^i}\right),\qquad i=1,\ldots,n. 
$$
By standard facts from Riemannian geometry, these vector fields satisfy the relations
$$
\abs{X_r}^2=1\qquad\text{and}\qquad\abs{X_i}^2=\sinh^2(r)\abs{\partial/\partial \theta^i}^2_{S^n},
$$
and moreover, at the point $x$ the vectors $X_r$, $X_1,\ldots,X_{m-1}$ form an orthogonal basis on $T_xC^m$.

The gradient of a radial function $f$ on $\mathbf{H}^{n+1}$ is collinear to $X_r$, and hence, lies in the tangent space $T_xC^m$. Thus, by a standard computation, see the proof of ~\cite[Lemma~2.7]{Ko21}, we conclude that the Laplacian of the restriction of $f$ to $C^m$ at $x$ is given by the formula
\begin{equation}
\label{lap:f}
-\Delta f(x)=\hess_x f(X_r,X_r)+\frac{1}{\sinh^2(r(x))}\sum_{i=1}^{m-1}\hess_x f(X_i,X_i)
\end{equation}
Computing the first term, we obtain
\begin{multline}
\label{hess:r}
\hess_x f(X_r,X_r)=X_r(D_xf(X_r))-D_xf(\nabla_{X_r}X_r)\\ =X_r(f'(r))-f'(r)\langle X_r,\nabla_{X_r}X_r\rangle=f''(r)-0,
\end{multline}
where we used that the product in the second term vanishes, since $X_r$ has unit length. Similarly, we obtain
$$
\hess_x f(X_i,X_i)=X_i(D_xf(X_i))-D_xf(\nabla_{X_i}X_i)=0-f'(r)\langle X_r,\nabla_{X_i}X_i\rangle,
$$
where the first term vanishes, since $f$ is radial. The product in the last term can be computed by Koszul's formula, which at the point $x$ gives
$$
\langle X_r,\nabla_{X_i}X_i\rangle=-\frac{1}{2}X_r\abs{X_i}^2=-\cosh(r)\sinh(r).
$$
Combining the last relations with formulae~\eqref{lap:f} and~\eqref{hess:r}, we arrive at the statement of the lemma. 
\end{proof}

Now choose a sequence $R_k\to +\infty$ such that $R_{k+1}>2R_k$, and denote by $\upsilon_k$ compactly supported smooth functions on the real line such that the supports $\spt\upsilon_k$ are mutually disjoint, and
\begin{multline}
\label{aprox:1}
\int\limits_{0}^{+\infty}\abs{\upsilon_{k}''(t)+(m-1)\coth(t)\upsilon_{k}'(t)+\lambda\upsilon_{k}(t)}^2\sinh^{m-1}(t)dt \\ \leqslant 2\int\limits_{R_k/2}^{R_k}\abs{\upsilon_{R_k}''(t)+(m-1)\coth(t)\upsilon_{R_k}'(t)+\lambda\upsilon_{R_k}(t)}^2\sinh^{m-1}(t)dt,
\end{multline}

\begin{equation}
\label{aprox:2}
\int\limits_{R_k/2}^{R_k}\abs{\upsilon_{R_k}}^2\sinh^{m-1}(t)dt\leqslant 2\int\limits_{0}^{+\infty}\abs{\upsilon_{k}}^2\sinh^{m-1}(t)dt.
\end{equation}
Such smooth functions $\upsilon_k$ can be constructed as approximations of $W^{2,2}$-functions $\upsilon_{R_k}$, for example, by using the mollification technique. Denote by $\phi_k$ the restriction of $\upsilon_k$ to the cone $C^m$. Then, by Lemma~\ref{laplace:cone} we obtain
$$
\int \abs{\Delta\phi_k-\lambda\phi_k}^2d\norm{C^m}=\omega(\Gamma^{m-1})\int\limits_{0}^{+\infty}\abs{\upsilon_{k}''(t)+(m-1)\coth(t)\upsilon_{k}'(t)+\lambda\upsilon_{k}(t)}^2\sinh^{m-1}(t)dt, 
$$
and the combination with relations~\eqref{aprox:1}--\eqref{aprox:2} yields
\begin{equation}
\label{f:cone}
\int \abs{\Delta\phi_k-\lambda\phi_k}^2d\norm{C^m}\leqslant\epsilon_k\int\abs{\phi_k}^2d\norm{C^m},
\end{equation}
where $\epsilon_k=4\epsilon_{R_k}$, and $\epsilon_{R_k}$ is given in Lemma~\ref{est}. Since $\epsilon_{R_k}\to 0+$, we conclude that the statement of the theorem holds for cones.

\medskip
\noindent
{\em Step~2: minimal varieties.} To get similar estimates for the test-functions $\upsilon_k$ restricted to $\reg\Sigma^m$ we need the following lemma.
\begin{lemma}
\label{laplace:current}
Let $\Sigma^m$ be a stationary locally rectifiable $m$-current from the set $\mathcal M(\Gamma^{m-1})$. Then for any $\varepsilon>0$ there exists $R>0$ such that for any smooth radial function $f$ that is supported in the complement of the hyperbolic ball $\mathbf{B}_{R}(p)$ the Laplacian of its restriction to $\spt\Sigma^m$ is given by the formula
$$
-\Delta f(x)=f''(r(x))+(m-1)\coth(r(x))f'(r(x)) +E(x),
$$
where the error term $E(x)$ satisfies the estimate
$$
\abs{E(x)}\leqslant\varepsilon\left(\abs{f''(r(x))}+\abs{f'(r(x))}\right),
$$
$r(x)=\dist(x,p)$, and $x\in\spt \Sigma^m$.
\end{lemma}
\begin{proof}
We follow the notation in the proof of Lemma~\ref{laplace:cone}. Below we always assume that $R>0$ is sufficiently large such that all points $x\in\spt\Sigma^m$ with $r(x)>R$ are contained in $\reg\Sigma^m$.

\smallskip
\noindent
{\em Partial case.} To illustrate the main idea, we first consider the case when the radial vector $(\partial/\partial r)$ lies in the tangent space $T_x\Sigma^m$ for a given point $x\in\spt\Sigma^m$. Our claim that in this case the error term in the formula for the Laplacian vanishes.

In more detail, consider the intersection of the geodesic sphere $\partial\mathbf{B}_r(p)$ with $\spt\Sigma^m$, where $r=r(x)$. Since $(\partial/\partial r)$ and $T_x\partial\mathbf{B}_r(p)$ span the tangent space $T_x\mathbf{H}^{n+1}$, we conclude that $\spt\Sigma^m$ and $\partial\mathbf{B}_r(p)$ intersect transversally around $x$, and $\spt\Sigma^m\cap\partial\mathbf{B}_r(p)$ is a submanifold of dimension $m-1$ in $\partial\mathbf{B}_r(p)$ around $x$. In geodesic spherical coordinates the point $x$ corresponds to the value $r(x)$ and a point $\xi\in S^n$, and the intersection  $\spt\Sigma^m\cap\partial\mathbf{B}_r(p)$ corresponds to a submanifold in $S^n$ via the exponential map. Choosing coordinates $(\theta^1,\ldots,\theta^n)$ around $\xi$ on $S^n$ in the same fashion as in the proof of Lemma~\ref{laplace:cone}, we obtain the collection of vector fields
$$
X_r=\frac{\partial}{\partial r},\qquad X_i=D\exp_p\left(\frac{\partial}{\partial\theta^i}\right),\qquad i=1,\ldots,n.
$$
They satisfy the same properties as in the proof of Lemma~\ref{laplace:cone}, and in particular, $X_r$, $X_1,\ldots,X_{m-1}$ form an orthogonal basis of the tangent space $T_x\Sigma^m$. Since $\spt\Sigma^m$ is a smooth submanifold around $x$ whose mean curvature vanishes, a standard computation, see~\cite{Ko21}, shows that the Laplacian of the restriction of $f$ to $\spt\Sigma^m$ at $x$ is given by formula~\eqref{lap:f}. The rest of the argument in this case follows the lines in the proof of Lemma~\ref{laplace:cone}.

\smallskip
\noindent
{\em General case.} Since $\spt\Sigma^m$ meets the boundary at infinity orthogonally, for a sufficiently large $r=r(x)$ the tangent spaces $T_x\Sigma^m$ and $T_x\partial\mathbf{B}_r(p)$ span $T_x\mathbf{H}^{n+1}$, and we conclude that $\spt\Sigma^m$ and $\partial\mathbf{B}_r(p)$ intersect transversally. Thus, the intersection $\spt\Sigma^m\cap\partial\mathbf{B}_r(p)$ is a submanifold in $\partial\mathbf{B}_r(p)$, and arguing as above we can construct vector fields $X_i$ in a neighbourhood of $x$ in $\mathbf{H}^{n+1}$, where $i=1,\ldots, n$, such that at the point $x$ they form an orthogonal system, and the first $(m-1)$ vectors $X_1,\ldots, X_{m-1}$ at $x$ lie in the tangent space $T_x\Sigma^m$. Now we define the vector $\tilde X_r\in T_x\Sigma^m$ as a unit vector that is orthogonal to these vectors and such that the product $\langle \tilde X_r,X_r\rangle$ is positive, where $X_r$ is the radial vector $\partial/\partial r$ used above. This vector $\tilde X_r$ can be extended to a vector field around $x$ such that
\begin{equation}
\label{def:xr}
\abs{\tilde X_r}^2=1\qquad\text{and}\qquad\langle\tilde X_r,X_i\rangle=0\quad\text{for all}\quad i=1,\ldots,m-1.
\end{equation}
In fact, since $\spt\Sigma^m$ is $C^1$-smooth up to boundary we can define the vector field $\tilde X_r$ in a neighbourhood of $\Gamma^{m-1}$ as unit vector field such that for any $x\in\spt\Sigma^m$ with a sufficiently large $r(x)$ it lies in $T_x\Sigma^m$ and is orthogonal to $\spt\Sigma^m\cap\partial\mathbf{B}_r(p)$. 

Next, we claim that the hyperbolic length $\abs{\tilde X_r-X_r}$ tends to zero as $r=r(x)$ tends to infinity. Indeed, denote by $\tilde Z_r$ and $Z_r$ vectors of unit Euclidean length obtained by scaling $\tilde X_r$ and $X_r$ respectively. Since $\spt\Sigma^m$ meets the boundary of the Euclidean ball $\mathbb B^{n+1}$ orthogonally, it is straightforward to see that the difference $\tilde Z_r-Z_r$ can be made arbitrarily small by choosing a sufficiently large $r=r(x)$. Now the claim follows from the observation that the hyperbolic length $\abs{\tilde X_r-X_r}$ and the Euclidean length of $\tilde Z_r-Z_r$ coincide. Thus, writing $\tilde X_r$ in the form
\begin{equation}
\label{decom}
\tilde X_r=\varphi X_r+Y_r,\qquad\text{where}\quad\langle X_r,Y_r\rangle=0,
\end{equation}
we conclude that the function $(1-\varphi^2)$ and the hyperbolic length $\abs{Y_r}$ tend to zero when $r=r(x)$ tends to infinity. Similarly to the computation used above, the Laplacian of the restriction of $f$ to $\Sigma^m$ at $x$ is now given by the formula
$$
%\label{lap:fe}
-\Delta f(x)=\hess_x f(\tilde X_r,\tilde X_r)+\frac{1}{\sinh^2(r(x))}\sum_{i=1}^{m-1}\hess_x f(X_i,X_i).
$$
The argument in the proof of Lemma~\ref{laplace:cone} shows that the last term gives the same contribution $(m-1)\coth(r)f'(r)$, and it remains to compute the first term only. Using the decomposition in~\eqref{decom}, we obtain
\begin{equation}
\label{hess:d}
\hess_x f(\tilde X_r,\tilde X_r)=\varphi^2\hess_x f(X_r,X_r)+2\varphi\hess_x f(X_r,Y_r)+\hess_x f(Y_r,Y_r).
\end{equation}
As the computation in the proof of Lemma~\ref{laplace:cone} shows, the first term in the last relation equals $\varphi^2f''(r)$, and the second term vanishes,
$$
\hess_x f(X_r,Y_r)=Y_r(f'(r))-f'(r)\langle X_r,\nabla_{Y_r}X_r\rangle=f''(r)\langle X_r,Y_r\rangle-0=0,
$$
where we used that $f$ is radial and $X_r$ has unit length. Similarly, the last term in relation~\eqref{hess:d} can be computed in the following fashion:
\begin{multline*}
\hess_x f(Y_r,Y_r)=Y_r(D_xf(Y_r))-D_xf(\nabla_{Y_r}Y_r)\\=Y_r(f'(r)\langle Y_r,X_r\rangle)-f'(r)\langle X_r,\nabla_{Y_r}Y_r\rangle=0-f'(r)\langle X_r,\nabla_{Y_r}Y_r\rangle.
\end{multline*}
Combining all relations above, we conclude that the error term $E(x)$ in the statement of the lemma, has the form
$$
E(x)=(1-\varphi^2)f''(r)-\langle X_r,\nabla_{Y_r}Y_r\rangle f'(r).
$$
Since we already know that the function $(1-\varphi^2)$ is small for a large $r=r(x)$, for a proof of the lemma it remains to show that so is the product $\langle X_r,\nabla_{Y_r}Y_r\rangle$.

Using relations~\eqref{def:xr}--\eqref{decom}, we see that the vector field $Y_r$ can be written in the form
$$
Y_r=\sum_{i=m}^n\alpha_iX_i.
$$
Since the vectors $X_i$ at the point $x$, where $i=1,\ldots,n$, form an orthogonal system, the values of the functions $\alpha_i$ at $x$ can be estimated in the following way:
$$
%\label{alpha:x}
\abs{\alpha_i(x)}=\frac{\abs{\langle Y_r,X_i\rangle_x}}{\abs{X_i}^2_x}\leqslant\frac{\abs{Y_r}_x}{\abs{X_i}_x}=\frac{\abs{Y_r}_x}{\sinh(r)},
$$
where $i=m,\ldots,n$. Since the radial vector $X_r$ is orthogonal to the vectors $X_i$'s, a direct computation gives
$$
\langle X_r,\nabla_{Y_r}Y_r\rangle=\sum_{i,j=m}^n\alpha_i\alpha_j\langle X_r,\nabla_{X_i}X_j\rangle,
$$
where the products $\langle X_r,\nabla_{X_i}X_j\rangle$ can be computed using Koszul's formula. In particular, at the fixed point $x$ we obtain 
$$
\langle X_r,\nabla_{X_i}X_j\rangle_x=-\delta_{ij}\cosh(r)\sinh(r),
$$
where $i,j=m,\dots,n$, $r=r(x)$, and $\delta_{ij}$ is the Kronecker delta. Combining the relations above, we finally arrive at the inequality
$$
\abs{\langle X_r,\nabla_{Y_r}Y_r\rangle_x}\leqslant(n-m)\coth(r)\abs{Y_r}^2_x.
$$
Since the hyperbolic length of $\abs{Y_r}$ can be made arbitrarily small by choosing a point $x$ with a sufficiently large $r=r(x)$, we conclude that so can be the left-hand side above.
\end{proof}

Now for a given $0<\varepsilon<1$ let $R>0$ be a real number that satisfies the conclusions of  Lemmas~\ref{vol} and~\ref{est}, as well as Lemma~\ref{laplace:current}. For a sequence $R_k\to +\infty$ such that $R_{k+1}>2R_k>2R$, we can choose compactly supported smooth functions  $\upsilon_k$ defined on the real line such that their supports $\spt\upsilon_k$ are mutually disjoint, they satisfy relations~\eqref{aprox:1}-\eqref{aprox:2}, and in addition, the relations
$$
\int\limits_{0}^{+\infty}\abs{\upsilon'_{k}}^2\sinh^{m-1}(t)dt\leqslant 2\int\limits_{R_k/2}^{R_k}\abs{\upsilon'_{R_k}}^2\sinh^{m-1}(t)dt,
$$
$$
\int\limits_{0}^{+\infty}\abs{\upsilon''_{k}}^2\sinh^{m-1}(t)dt\leqslant 2\int\limits_{R_k/2}^{R_k}\abs{\upsilon''_{R_k}}^2\sinh^{m-1}(t)dt.
$$
Combining these relations with inequalities in Lemma~\ref{est} and relation~\eqref{aprox:2}, we obtain
\begin{equation}
\label{aprox:3}
\int\limits_{0}^{+\infty}\abs{\upsilon'_{k}}^2\sinh^{m-1}(t)dt\leqslant 4C_*\int\limits_{0}^{+\infty}\abs{\upsilon_{k}}^2\sinh^{m-1}(t)dt,
\end{equation}
\begin{equation}
\label{aprox:4}
\int\limits_{0}^{+\infty}\abs{\upsilon''_{k}}^2\sinh^{m-1}(t)dt\leqslant 4C_*\int\limits_{0}^{+\infty}\abs{\upsilon_{k}}^2\sinh^{m-1}(t)dt.
\end{equation}
Denote by $\tilde\phi_k$ the restriction of the function $\upsilon_k$ to the regular set $\reg\Sigma^m$. Then, by Lemma~\ref{laplace:current} we obtain
\begin{multline}
\label{f:source}
\int \abs{\Delta\tilde\phi_k-\lambda\tilde\phi_k}^2d\norm{\Sigma^m}\leqslant
3\int \abs{\upsilon''_k+(m-1)\coth(\cdot)\upsilon'_k+\lambda\upsilon_k}^2 d\norm{\Sigma^m}\\
+3\varepsilon^2\int \abs{\upsilon''_k}^2+\abs{\upsilon'_k}^2 d\norm{\Sigma^m}.
\end{multline}
To estimate the first integral on the right hand-side above we use Lemmata~\ref{vol} and~\ref{laplace:cone}, and already obtained estimate~\eqref{f:cone} for cones. Combining all these ingredients, we get
\begin{multline*}
\int \abs{\upsilon''_k+(m-1)\coth(\cdot)\upsilon'_k+\lambda\upsilon_k}^2 d\norm{\Sigma^m}\\
\leqslant(1+\varepsilon)\int \abs{\upsilon''_k+(m-1)\coth(\cdot)\upsilon'_k+\lambda\upsilon_k}^2 d\norm{C^m}=(1+\varepsilon)\int \abs{\Delta\phi_k-\lambda\phi_k}^2d\norm{C^m}\\
\leqslant\varepsilon_k(1+\varepsilon)\int\abs{\phi_k}^2d\norm{C^m}
\leqslant\varepsilon_k\frac{1+\varepsilon}{1-\varepsilon}\int\abs{\tilde\phi_k}^2d\norm{\Sigma^m}.
\end{multline*}
Here, as in Step~1, we denote by $\phi_k$ the restriction of the radial function $\upsilon_k$ to the cone $C^m$, and $\varepsilon_k=4\varepsilon_{R_k}$, where $\varepsilon_R$ is given by the relation in Lemma~\ref{est}. To estimate the second integral on the right hand-side of relation~\eqref{f:source}, we use integral inequalities~\eqref{aprox:3} and~\eqref{aprox:4} for the first two derivatives of $\upsilon_k$. In more detail, we can bound the $L^2$-norm of $\upsilon'_k$ on $\reg\Sigma^m$ in the following fashion:
\begin{multline*}
\int\abs{\upsilon'_k}^2 d\norm{\Sigma^m}\leqslant (1+\varepsilon)\int\abs{\upsilon'_k}^2 d\norm{C^m}
\\=(1+\varepsilon)\omega(\Gamma^{m-1})\int\limits_0^{+\infty}\abs{\upsilon'_k(t)}^2\sinh^{m-1}(t)dt\\
\leqslant 4C_*(1+\varepsilon)\omega(\Gamma^{m-1})\int\limits_0^{+\infty}\abs{\upsilon_{k}(t)}^2\sinh^{m-1}(t)dt\\ =4C_*(1+\varepsilon)\int\abs{\upsilon_k}^2 d\norm{C^m}
\leqslant 4C_*\frac{1+\varepsilon}{1-\varepsilon}\int\abs{\tilde\phi_k}^2 d\norm{\Sigma^m}.
\end{multline*}
Similarly, we have
$$
\int\abs{\upsilon''_k}^2 d\norm{\Sigma^m}\leqslant 4C_*\frac{1+\varepsilon}{1-\varepsilon}\int\abs{\tilde\phi_k}^2 d\norm{\Sigma^m}.
$$
Combining these estimates with relation~\eqref{f:source}, we finally obtain
$$
\int \abs{\Delta\tilde\phi_k-\lambda\tilde\phi_k}^2d\norm{\Sigma^m}\leqslant
(3\varepsilon_k+24C_*\varepsilon^2)\frac{1+\varepsilon}{1-\varepsilon}\int\abs{\tilde\phi_k}^2 d\norm{\Sigma^m}.
$$
Since $0<\epsilon<1$ is arbitrary and $R_k\to+\infty$, choosing a subsequence $\tilde\phi_{k_\ell}$, we may assume that
$$
\int \abs{\Delta\tilde\phi_{k_\ell}-\lambda\tilde\phi_{k_\ell}}^2d\norm{\Sigma^m}\leqslant
(3\varepsilon_{k_\ell}+24C_*/\ell^{2})\frac{1+1/\ell}{1-1/{\ell}}\int\abs{\tilde\phi_{k_\ell}}^2 d\norm{\Sigma^m}
$$
for all integers $\ell>1$. Since $\varepsilon_k=4\varepsilon_{R_k}\to 0+$, by Lemma~\ref{est},  we are done.
\qed

\subsection{Proof of Theorem~\ref{t1}}
The proof of Theorem~\ref{t1} follows closely the lines in the proof of Theorem~\ref{t2} for the higher codimension area-minimising current. In more detail, by Proposition~\ref{p:hl} any area-minimising locally rectifiable $n$-current $\Sigma^n$ that is asymptotic to $\Gamma^{n-1}$ lies in the class $\mathcal M(\Gamma^{n-1})$. Thus, the argument in the proof of Theorem~\ref{t2} carries over directly, and shows that the spectrum of the Friedrichs extension of the Laplacian is indeed the interval $[(n-1)^2/4,+\infty)$.
\qed

\section{Final remarks}
In this section we collect a few useful remarks, reflecting on the proof of Theorem~\ref{t2}.

\begin{remark}
It might be useful to note that together with Remark~\ref{rem:cones} the argument in the proof of Theorem~\ref{t2} shows that the Laplacian spectrum of any cone asymptotic to $\Gamma^{m-1}$ is precisely the interval $[(m-1)^2/4,+\infty)$. Indeed, the statement holds for the cone $C^m$, given by relation~\eqref{cone:def}, that is used in the proof. Its singular point $p$ corresponds to the origin of a unit Euclidean ball $\mathbb{B}^{n+1}$ via the Poincar\'e model. Since the isometry group of $\mathbf{H}^{n+1}$ acts transitively on it, we conclude that the statement holds for all cones.
\end{remark}

\begin{remark}
\label{r:c1}
Our argument in the proofs of Theorems~\ref{t1} and~\ref{t2} applies to any stationary current that lies in the class $\mathcal M(\Gamma^{m-1})$, that is if it does not contain singularities near the boundary, is $C^1$-smooth up to the boundary, and meets the boundary orthogonally. In co-dimension one the area-minimising hypothesis is needed only to ensure that these properties hold, see Proposition~\ref{p:hl}. Let us point out that if for a stationary current $\Sigma^m$ its support $\spt\Sigma^m$ is a $C^1$-smooth submanifold near and up to boundary, then it automatically meets the boundary at infinity orthogonally. Indeed, using a standard argument based on the first variation, see the proof of~\cite[Lemma~5]{An82}, or an appropriate version of the maximum principle~\cite{JT03,SW89}, one can foliate $\mathbf{H}^{n+1}$ by totally geodesic subspaces that serve as barriers for $\Sigma^m$. Then, following the idea in~\cite[Section~1]{HL87}, one can show that the tangent cone at the boundary point $z\in\Gamma^{m-1}$ is contained in the product $T_z\Gamma^{m-1}\times [0,+\infty)$, where we identify $\mathbf{H}^{n+1}$ with the upper half-space. We refer to~\cite{Ko25} where details of this argument can be found.

In particular, the claim above shows that our argument in the proof of Theorem~\ref{t2} applies to $C^1$-smooth up to boundary at infinity minimal submanifolds, sharpening Theorem~\ref{t0} due to~\cite{LMMV}, where the assumption that they are $C^2$-smooth up to boundary is used. The last hypothesis is used in~\cite{LMMV} to deduce that the corresponding submanifold has finite total curvature, which is an essential step in the proof of Theorem~\ref{t0} in that paper. The latter fails even for simplest singular minimal submanifolds, such as cones.
\end{remark}

\begin{remark}
Let $\Sigma^m$ be an area-minimising $m$-current in $\mathbf{H}^{n+1}$ that satisfies the hypotheses of Theorem~\ref{t1} or Theorem~\ref{t2}. As was discussed in Section~\ref{prems}, to our knowledge it is unknown whether the Laplace operator on $\Sigma^m$, defined on smooth compactly supported functions on $\reg\Sigma^m$, has a unique self-adjoint extension. For this reason we had to specify that we consider the Friedrichs extension, that is associated with the closure of the Dirichlet energy. In particular, this gives the variational characterisation of the bottom of the spectrum, which is used to prove the version of McKean's inequality, that is inequality~\eqref{MK}. On the other hand, in the proof of the statement that any $\lambda>(m-1)^2/4$ lies in the spectrum we used compactly supported smooth test-functions only, and hence, the argument holds for arbitrary self-adjoint extensions.

In summary, our argument in Section~\ref{proofs} shows that the interval $[(m-1)^2/4,+\infty)$ lies in the spectrum of any self-adjoint extension of the Laplacian, and it coincides with the whole spectrum of the Friedrichs extension. 
\end{remark}

%\appendix
%\section{}


{\small
\begin{thebibliography}{99}
\addcontentsline{toc}{section}{References}

\bibitem{An82} Anderson,~M.~T. {\it Complete minimal varieties in hyperbolic space.} Invent. Math. {\bf 69} (1982), 477--494.

\bibitem{BM07} Bessa,~G.~P., Montenegro,~J.~F. {\em An extension of Barta's theorem and geometric applications.} Ann. Global Anal. Geom. {\bf 31} (2007), 345--362.

\bibitem{Cha} Chavel,~I. {\it Eigenvalues in Riemannian geometry} Pure Appl. Math., {\bf 115}
Academic Press, Inc., Orlando, FL, 1984. xiv+362 pp.

\bibitem{CLY} Cheng,~S.-Y., Li,~P., Yau,~S.-T. {\it Heat equations on minimal submanifolds and their applications.} Amer. J. Math. {\bf 106} (1984), 1033--1065.

\bibitem{BD} Davies,~E.~B. {\em Spectral theory and differential operators.} Cambridge Studies in Advanced Mathematics, 42. Cambridge University Press, Cambridge, 1995. x+182 pp.

\bibitem{FS13} Fraser,~A., Schoen,~R. {\em Minimal surfaces and eigenvalue problems.} Contemp. Math., 599, American Mathematical Society, Providence, RI, 2013, 105--121.

\bibitem{GM12} Giaquinta,~M., Martinazzi,~L. {\em An introduction to the regularity theory for elliptic systems, harmonic maps and minimal graphs.} Appunti. Sc. Norm. Super. Pisa (N. S.), {\bf 11}. Edizioni della Normale, Pisa, 2012, xiv+366 pp.

\bibitem{HL87} Hardt,~R., Lin,~F.-H. {\em Regularity at infinity for area-minimising hypersurfaces in hyperbolic space.} Invent. Math {\bf 88} (1987), 217-224.

\bibitem{JT03} Jorge,~L.~P., Tomi,~F. {\it The barrier principle for minimal submanifolds of arbitrary codimension.} Ann. Global Anal. Geom. {\bf 24} (2003), 261--267.

\bibitem{Ko14} Kokarev,~G. {\em Variational aspects of Laplace eigenvalues on Riemannian surfaces.} Adv. Math. {\bf 258} (2014), 191--239.

\bibitem{Ko20} Kokarev,~G. {\em Bounds for Laplace eigenvalues of K\"ahler metrics.} Adv. Math. {\bf 365} (2020), 107061, 22 pp.

\bibitem{Ko21} Kokarev,~G. {\it Berger's inequalities in the presence of upper sectional curvature bound.} IMRN (2022), 11262--11303.

\bibitem{Ko25} Kokarev,~G. {\em On the essential spectra of submanifolds in the hyperbolic space.} Pure and Applied Funct. Anal., to appear.

\bibitem{LT95} Li,~P., Tian,~G. {\em On the heat kernel of the Bergmann metric on algebraic varieties.} Jour. Amer. Math. Soc. {\bf 8} (1995), 857--877.

\bibitem{LMMV} Lima,~B.~P., Mari,~L., Montenegro,~J.~F., Viera,~F.~B. {\em Density and spectrum of minimal submanifolds in space forms.} Math. Ann. {\bf 366} (2016), 1035--1066.

\bibitem{L89a} Lin,~F.-H. {\em On the Dirichlet problem for minimal graphs in the hyperbolic space.} Invent. Math. {\bf 96} (1989), 593--612.

\bibitem{L89b} Lin,~F.-H. {\em Asymptotic behavior of area-minimising currents in hyperbolic space.} Comm. Pure Appl. Math. {\bf 42} (1989), 229--242.

\bibitem{Ma86} Markvorsen,~S. {\it On the heat kernel comparison theorems for minimal submanifolds.} Proc. Amer. Math. Soc. {\bf 97} (1986), 479--482.

\bibitem{McK70} McKean,~H.~P. {\em An upper bound to the spectrum of $\Delta$ on a manifold of negative curvature.} J. Differential Geom. {\bf 4} (1970), 359--366.

\bibitem{FM} Morgan, ~F. {\it Geometric measure theory. A beginner's guide.} Academic Press, Inc., Boston, MA, 1988, viii+145 pp.

\bibitem{LS} Simon,~L. {\it Lectures on Geometric Measure Theory.} Proc. Centre Math. Anal. Austral. Nat. Univ., {3}. Australian National University, Centre for Mathematical Analysis, Canberra, 1983, vii+272 pp.

\bibitem{SW89} Solomon,~B., White,~B. {\it A strong maximum principle for varifolds that are stationary with respect to even parametric elliptic functionals.} Indiana Univ. Math. J. {\bf 38} (1989), 683--691.

\bibitem{MT} Taylor,~M.~E. {\em Partial Differential Equations~II. Qualitative Studies of Linear Equations.} Appl. Math. Sci., 116. Springer, New York, 2011, xxii+614 pp.

\end{thebibliography}
}
\end{document}